\documentclass[11pt]{amsart}
\usepackage{amsfonts, amssymb, amsmath, amsthm, hyperref, color, float,enumerate}

\theoremstyle{plain}
   \newtheorem{teo}{Theorem}
   
   \newtheorem{lema}[teo]{Lemma}

\theoremstyle{definition}
   \newtheorem{defi}{Definition}
\theoremstyle{remark}
 \newtheorem{obs}{Remark}
 
 \newtheorem{afirmacion}{Claim}

\numberwithin{equation}{section}
\allowdisplaybreaks

\newcommand{\R}{\mathbb{R}} 
\newcommand{\N}{\mathbb{N}} 
\newcommand{\Z}{\mathbb{Z}} 
\newcommand{\e}{\varepsilon} 
\newcommand{\s}[2]{\sum_{#1}^{#2}} 
\newcommand{\norm}[1]{\left\|#1\right\|} 

\begin{document}
	
	\title[Improvements on Sawyer type estimates]{Improvements on Sawyer type estimates  for generalized maximal functions}

	\author[F. Berra]{Fabio Berra}
	\address{CONICET and Departamento de Matem\'{a}tica (FIQ-UNL),  Santa Fe, Argentina.}
	\email{fberra@santafe-conicet.gov.ar}
	
	\author[M. Carena]{Marilina Carena}
	\address{CONICET (FIQ-UNL) and Departamento de Matem\'{a}tica (FHUC-UNL),  Santa Fe, Argentina.}
	\email{marilcarena@gmail.com}
	
	\author[G. Pradolini]{Gladis pradolini}
	\address{CONICET and Departamento de Matem\'{a}tica (FIQ-UNL),  Santa Fe, Argentina.}
	\email{gladis.pradolini@gmail.com}

	\thanks{The authors were supported by CONICET and UNL}
	
	
	
	\begin{abstract}
		In this paper we prove mixed inequalities for the maximal operator $M_\Phi$, for general Young functions $\Phi$ with certain additional properties, improving and generalizing some previous estimates for the Hardy-Littlewood maximal operator proved by E. Sawyer. We show that given $r\geq 1$, if $u,v^r$ are weights belonging to the $A_1$-Muckenhoupt class and $\Phi$ is a Young function as above, then the inequality
		\[uv^r\left(\left\{x\in \R^n: \frac{M_\Phi(fv)(x)}{v(x)}>t\right\}\right)\leq C\int_{\R^n}\Phi\left(\frac{|f(x)|}{t}\right)u(x)v^r(x)\,dx\]
		holds for every positive $t$.
		
		A motivation for studying these type of estimates is to find an alternative way to prove the boundedness properties of $M_\Phi$. Moreover, it is well-known that for the particular case $\Phi(t)=t(1+\log^+t)^m$ with $m\in\N$ these maximal functions control, in some sense, certain operatos in Harmonic Analysis.
	\end{abstract}

	\maketitle
	
	\section*{Introduction}
	
	In \cite{Muckenhoupt_Wheeden_77}, B. Muckenhoupt and R. Wheeden proved certain weighted weak estimates that involved the Hardy-Littlewood maximal operator or the Hilbert transform. More precisely, they proved that given $1\leq p<\infty$ and $w\in A_p$, there exists a positive constant $C$ such that the inequality
	\begin{equation}\label{estimacion M-W}
	\left|\left\{x\in \R: T f(x)w^{1/p}(x)>t\right\}\right|\leq \frac{C}{t^p}\int_{\R} |f(x)|^pw(x)\,dx
	\end{equation} 
	holds for every positive $t$, where $T$ is either of the two operators mentioned above. These type of inequalities were studied as a motivation to prove some two weighted norm inequalities like those that appear in \cite{M-W-76}. 
	
	The difference between these estimates and the classical weak type inequalities  is that we must handle with level sets of product of functions, and this fact suggests that classical covering lemmas or decomposition techniques would not apply directly. In \cite{Muckenhoupt_Wheeden_77}, the authors use a special classification of intervals, and the inequality \eqref{estimacion M-W} follows from an estimation of the measure of certain subsets of them, called ``principal intervals''. 
	
	Inspired in this paper, few years later E. Sawyer proved in \cite{Sawyer} that, if $u,v$ are weights belonging to the $A_1$-Muckenhoupt class and $M$ is the Hardy-Littlewood maximal operator, then the estimate
	\begin{equation}\label{estimacion sawyer}
	uv\left(\left\{x\in \R: \frac{M(fv)(x)}{v(x)}>t\right\}\right)\leq C\int_\R |f(x)|u(x)v(x)\,dx 
	\end{equation} 
	holds for every positive $t$. This last inequality can be seen as the weak (1,1) type of the auxiliary operator $S$ defined by $Sf(x)=M(fv)(x)v^{-1}(x)$ with respect to the measure $d\mu=uv\,dx$. Furthermore, it can be used to give an alternative proof of the boundedness of the Hardy-Littlewood maximal operator in $L^p(w)$, for $1<p<\infty$ and $w\in A_p$, proved by B. Muckenhoupt in \cite{Muckenhoupt-72}. 
	

	Later on, in \cite{CU-M-P} the authors extended inequality \eqref{estimacion sawyer} to $\R^n$ and for both the Hardy-Littlewood maximal function and Calder\'on-Zygmund operators (CZOs). They considered two pair of conditions on the weights involved: $u,v\in A_1$ and $u\in A_1, v\in A_{\infty}(u)$. For the first case they follow similar ideas as in \cite{Sawyer}. The second condition is instead more ``suitable'' in the sense that the product $uv$ is an $A_\infty$-weight and therefore some classical techniques like Calder\'on-Zygmund decomposition can be applied.  The main idea in this work is to obtain the corresponding mixed estimate for the dyadic Hardy-Littlewood maximal operator, $M_{\mathcal{D}}$, and then obtain an analogous result for $M$ by extrapolation techniques. 
	The corresponding estimate for CZOs is achieved in a similar way. 
	
	Recently, in \cite{Li-Ombrosi-Perez} the authors extended the estimates given in \cite{CU-M-P} to a more general case. More precisely, they proved that if $\mathcal{T}$ is either the Hardy-Littlewood maximal operator or a CZO, $u\in A_1$ and $v\in A_\infty$ then the estimate
	\begin{equation}\label{extension v en A infinito}
 uv\left(\left\{x\in \R^n: \frac{|\mathcal{T}(fv)(x)|}{v(x)}>t\right\}\right)\leq C\int_{\R^n} |f(x)|u(x)v(x)\,dx 
	\end{equation}
	holds for every positive $t$. 
	
	Observe that \eqref{estimacion M-W} can be obtained as a direct consequence of \eqref{extension v en A infinito} by taking $u\in A_1$, $v=u^{-1}\in A_2\subset A_\infty$ and $\tilde f=fv$.
	
	Then, a natural question is whether such estimates remain true for a more general class of maximal operators, which control in some sense classical operators from Harmonic Analysis. For example, it is well-known that certain maximal operators associated to the Young function $\varphi(t)=t(1+\log^+t)^m$ control the higher order commutators of CZOs. In this direction, in \cite{bcp}, the authors proved mixed weak estimates  in $\R^n$ for weights $u$  and $v$, where $u$ is arbitrary but $v=|x|^{\beta}$ with $\beta<-n$. Concretely, we proved that   
	
	\begin{equation*}
	uw\left(\left\{x\in \R^n: \frac{M_{\Phi_0}(fv)(x)}{v(x)}>t\right\}\right)\leq C\int_{\R^n}\Phi_0\left(\frac{|f(x)|v(x)}{t}\right)Mu(x)\,dx
	\end{equation*}
	holds for every positive $t$, where $w=1/\Phi_0(v^{-1})$, $\Phi_0(t)=t^r(1+\log^+t)^{\delta}$, with $r\geq 1$ and $\delta\geq0$. 
	
	Later on, in \cite{Berra} the author showed that a similar behavior occurs in the case $u, v^r\in A_1$, that is,   
	\begin{equation}\label{mixta_para_M_Phi_version1}
	uw\left(\left\{x\in \R^n: \frac{M_{\Phi_0}(fv)(x)}{v(x)}>t\right\}\right)\leq C\int_{\R^n}\Phi_0\left(\frac{|f(x)|v(x)}{t}\right)u(x)\,dx,
	\end{equation}
	where $w$ and $\Phi_0$ are as above.
	
	
	A motivation for studying these type of estimates is to find an alternative way to prove the boundedness of the operator $M_{\Phi}$.  Although in \cite{Berra} was established that \eqref{mixta_para_M_Phi_version1} extends the estimates in \cite{CU-M-P} not only for $M$ but also for $M_r$, this inequality turns out to be non-homogeneous, and even when $v^r\in A_1$, the resulting weight $w$ might not. This fact forbids us to use the result in order to achieve such an alternative proof. Since it is known that the operator $M_{\Phi_0}$ is bounded in $L^p(w)$ for $r<p<\infty$ and $w\in A_{p/r}$ (see \cite{B-D-P}), which is the same condition for the boundedness of the operator $M_r$, an interesting question is if \eqref{mixta_para_M_Phi_version1} could be improved. 
	
	In this paper we answer this question positively. Moreover, we prove mixed weak estimates for the operator $M_{\Phi}$ for general Young functions $\Phi$ with some additional properties, improving and generalizing the previous estimates. 
	Given $r\geq 1$, we define the class $\mathcal{F}_r$ as the set of all the Young functions $\Phi$ that have a lower type $r$, are submultiplicative and verify that there exist constants $C_0>0$, $\delta\geq0$ and $t_0\geq 1$ such that
	\begin{equation}\label{propiedad acotacion de Phi}
	\frac{\Phi(t)}{t^r}\leq C_0 (\log t)^\delta, \quad\textrm{ for }t\geq t_0.
	\end{equation}
    Concretely, we have the following result.
	
	\begin{teo}\label{teorema principal}
Let $r\geq 1$ and $\Phi\in \mathcal{F}_r$. If $u,v^r$ are weights belonging to the $A_1$-Muckenhoupt class, then there exists a positive constant $C$ such that 
\[uv^r\left(\left\{x\in \R^n: \frac{M_\Phi(fv)(x)}{v(x)}>t
\right\}\right)\leq C\int_{\R^n}\Phi\left(\frac{f}{t}\right)uv^r\,dx\]
holds for every $t>0$ and every bounded function with compact support.
	\end{teo} 

\begin{obs}
	The family of Young functions $\Phi_0(t)=t^r(1+\log^+ t)^{\delta}$ with $r\geq 1$ and $\delta\geq 0$ belongs to $\mathcal{F}_r$. Moreover, for these type of functions the aforementioned result is an improvement of \eqref{mixta_para_M_Phi_version1} in two senses: the inequality involved  is homogeneous in $v$ and on the other hand it is suitable in order to obtain the boundedness of $M_\Phi$. (see $\S$\ref{seccion: interpolacion y aplicaciones}).
	
	Many other examples can be given. As we said above, the functions $\Phi_0(t)=t^r(1+\log^+t)^{\delta}$ belong to $\mathcal{F}_r$. Also, we can  consider the function defined by
	\[\Phi(t)=\left\{\begin{array}{ccl}
	t^q& \textrm{ if } & 0\leq t\leq 1,\\
	t^r\log(e+\log(e+t))^{\delta} & \textrm{ if } & t>1.
	\end{array}
	\right.,\]
	where $q\geq r$, and $\delta\geq 0$. These functions are also in $\mathcal{F}_r$. This example includes combination of power functions when $\delta=0$, or power and $L \log\log L$ functions, if $\delta>0$. 
\end{obs}

The remainder of this paper is organized as follows. In $\S$1 we give the preliminaries and basic definitions. In $\S$2 we prove the main result and finally,  in $\S$3, we use interpolation techniques for modular type inequalities and the main result to give an alternative proof of the boundedness of the maximal operator $M_\Phi$.

\section{Preliminaries and basic definitions} 

We shall use the notation $A\lesssim B$ to mean that there exists a positive constant $C$ such that $A\leq C B$. The constant $C$ may change on each occurrence. We say that $A\approx B$ if $A\lesssim B$ and $B\lesssim A$.

Given a function $\varphi$, we will say that $f\in L^{\varphi}_{loc}$ if $\varphi(|f|)$ is locally integrable. When we consider the function $\varphi(t)=t$, the corresponding space is the usual $L^1_{loc}$.

By a \emph{weight} $w$ we mean function that is locally integrable, positive and finite in almost every $x$. Given $1<p<\infty$, the $A_p$-Muckenhoupt class is defined to be the set of weights $w$ that verify 
\[\left(\frac{1}{|Q|}\int_Qw\,dx\right)\left(\frac{1}{|Q|}\int_Qw^{1-p'}\,dx\right)^{p-1}\leq C,\]
for some positive constant $C$ and for every cube $Q\subseteq \R^n$. By a cube $Q$ we understand a cube in $\R^n$ with sides parallel to the coordinate axes.
If $p=1$ we say that $w\in A_1$ if there exists a positive constant $C$ such that for every cube $Q$
\[\frac{1}{|Q|}\int_Qw\,dx\leq C\inf_Qw.\]

Finally, the $A_\infty$ class is defined as the collection of all the $A_p$ classes, that is, $A_\infty=\bigcup_{p\geq 1}A_p$. It is well known that the $A_p$ classes are increasing on $p$, that is, if $p\leq q$ then $A_p\subseteq A_q$.  For further details and other properties of weights see \cite{javi} or \cite{grafakos}.

There are many conditions that characterize $A_\infty$. In this paper we will use the following one: $w\in A_\infty$ if there exist positive constants $C$ and $\e$ such that, for every cube $Q\subseteq \R^n$ and every measurable set $E\subseteq Q$ the condition
\[\frac{w(E)}{w(Q)}\leq C\left(\frac{|E|}{|Q|}\right)^{\e}\]
holds, where $w(E)=\int_E w$.

The smallest constants $C$ for which the corresponding inequalities above hold are denoted by $[w]_{A_p}$, $1\leq p\leq \infty$ and called the $A_p$ constants of $w$.

An important property of  Muckenhoupt weights is the \emph{reverse H\"{o}lder condition}. This means that given
$w\in A_p$, for some $1\leq p<\infty$, there exist positive constants $C$
and $s>1$ that depend only on the dimension $n$, $p$ and
$[w]_{A_p}$, such that 
\begin{equation*}
\left(\frac{1}{|Q|}\int_Q w^s(x)\,dx\right)^{1/s}\leq
\frac{C}{|Q|}\int_Q w(x)\,dx
\end{equation*}
for every cube $Q$. We write $w\in
\textrm{RH}_s$ to point out that the inequality above holds,
and we denote by $[w]_{\textrm{RH}_s}$ the smallest constant $C$ for
which this condition holds. A weight $w$ belongs to RH$_\infty$ if
there exists a positive constant $C$ such
that\begin{equation*}
\sup_Q w\leq\frac{C}{|Q|}\int_Q w,
\end{equation*}
for every $Q\subset \R^n$. Let us observe that
$\textrm{RH}_\infty\subseteq \textrm{RH}_s\subseteq \textrm{RH}_q$,
for every $1<q<s\leq \infty$.\\

Given a locally integrable function $f$, the \emph{Hardy-Littlewood maximal operator} is defined by
\[Mf(x)=\sup_{Q\ni x}\frac{1}{|Q|}\int_Q |f(y)|\,dy.\]

We say that $\varphi:[0,\infty)\to[0,\infty)$ is a \emph{Young function} if it is convex, increasing, $\varphi(0)=0$ and $\varphi(t)\to\infty$ when $t\to\infty$. Given a Young function $\varphi$, the maximal operator $M_\varphi$ is defined, for $f\in L^\varphi_{\textit{loc}}$, by
\[M_\varphi f(x)=\sup_{Q\ni x}\norm{f}_{\varphi,Q},\]
where $\norm{f}_{\varphi,Q}$ denotes the \emph{Luxemburg type average} of the function $f$ in the cube $Q$, which is defined as follows
\[\norm{f}_{\varphi,Q}=\inf\left\{\lambda>0 : \frac{1}{|Q|}\int_Q\varphi\left(\frac{|f(y)|}{\lambda}\right)\,dy\leq 1 \right\}.\]

We also define the \emph{weighted Luxemburg type average}
 $\norm{f}_{\varphi,Q,w}$ by
 \[\norm{f}_{\varphi,Q,w}=\inf\left\{\lambda>0 : \frac{1}{w(Q)}\int_Q\varphi\left(\frac{|f(y)|}{\lambda}\right)w(y)\,dy\leq 1 \right\}.\]

If $w\in A_\infty$ then $d\mu=w(x)\,dx$ is a doubling measure. Thus, by following the same arguments as in the result of Krasnosel'ski{\u\i} and Ruticki{\u\i} (\cite{KR}, see also \cite{raoren}) we can get that
\begin{equation}\label{equivalencia normal Luxemburo con infimo}
\|f\|_{\Phi,Q,w}\approx\inf_{\tau>0}\left\{\tau+\frac{\tau}{w(Q)}\int_{Q}\Phi\left(\frac{|f|}{\tau}\right)w\,dx\right\}.
\end{equation}

Given a Young function $\varphi$,
we use $\tilde{\varphi}$ to denote the
\emph{complementary Young function} associated to $\varphi$, defined
for $t\geq 0$ by
\[\tilde{\varphi}(t)=\sup\{ts-\varphi(s):s\geq 0\}.\]
It is well known in the literature that $\tilde{\varphi}$ satisfies
\[t\leq \varphi^{-1}(t)\tilde{\varphi}^{-1}(t)\leq 2t,\quad\forall t>0,\]
where $\varphi^{-1}$ denotes the \emph{generalized inverse} of $\varphi$, defined by
\[\varphi^{-1}(t)=\inf\{s>0: \varphi(s)>t\}.\]

For a Muckenhoupt weight $w$ and a Young function $\varphi$ we have the following
\emph{generalized H\"{o}lder inequality}
\begin{equation}\label{Hölder generalizada}
\frac{1}{w(Q)}\int_Q |fg| w\,dx\leq C \|f\|_{\varphi,Q,w}\|g\|_{\bar{\varphi},Q,w}.
\end{equation}

We say that a Young function $\varphi$ has \textit{lower type} $q$, $0<q<\infty$ if there exists a positive constant $C_q$ such that
\[\varphi(st)\leq C_qs^q\varphi(t),\]
for every $0<s\leq 1$ and $t>0$. As an immediate consequence of this definition we have that, if $\varphi$ has lower type $q$ then $\varphi$ has lower type $p$, for every $0<p<q$.

Given a Young function $\varphi$ and $1<p<\infty$ we say that $\varphi$ satisfies the $B_p$ condition and denote it by $\varphi\in B_p$ if there exists a positive constant $c$ such that 
\begin{equation}\label{condicion Bp}
\int_c^\infty \frac{\varphi(t)}{t^p}\,\frac{dt}{t}<\infty.
\end{equation}


A \emph{dyadic grid} $\mathcal{D}$ will be understood as a collection of cubes of $\R^n$ that satisfies the following properties:
\begin{enumerate}
	\item every cube $Q$ in $\mathcal{D}$ has side length $2^k$, for some $k\in\Z$;
	\item if $P\cap Q\neq\emptyset$ then $P\subseteq Q$ or $Q\subseteq P$;
	\item $\mathcal{D}_k=\{Q\in\mathcal{D}: \ell(Q)=2^k\}$ is a partition of $\R^n$ for every $k\in\Z$, where $\ell(Q)$ denotes the side length of $Q$.
\end{enumerate}

To a given dyadic grid $\mathcal{D}$ we can associate the corresponding maximal operator $M_{\varphi,\mathcal{D}}$ defined similarly as above, but where the supremum is taken over all cube in $\mathcal{D}$. When $\varphi(t)=t$, we will simply denote  this operator with $M_{\mathcal{D}}$.

\medskip

The next result will be useful in our estimates. A proof can be found in \cite{Okikiolu}.

\begin{teo}\label{teo_control_diadico}
	There exist dyadic grids $\mathcal{D}^{(i)}$, $1\leq i\leq 3^n$ such that, for every cube $Q\subseteq\R^n$, there exist $i$ and $Q_0\in \mathcal{D}^{(i)}$ satisfying $Q\subseteq Q_0$ and $\ell(Q_0)\leq 3\ell(Q)$.
\end{teo}

From the theorem above, we obtain that
\begin{equation}\label{control diadico de M phi}
M_{\Phi}f(x)\leq C\s{i=1}{3^n}M_{\Phi,\mathcal{D}^{(i)}}f(x).
\end{equation}
Indeed, fix $x\in\R^n$ and $Q$ a cube containing $x$. By Theorem~\ref{teo_control_diadico} we have a dyadic grid $\mathcal{D}^{(i)}$ and $Q_0\in \mathcal{D}^{(i)}$ with the desired properties. Then,
\[\frac{1}{|Q|}\int_Q \Phi\left(\frac{|f(y)|}{\norm{f}_{\Phi,Q_0}}\right)\,dy \leq\frac{|Q_0|}{|Q|}\frac{1}{|Q_0|}\int_{Q_0}\Phi\left(\frac{|f(y)|}{\norm{f}_{\Phi,Q_0}}\right)\,dy
\leq 3^n,\]
so
\[\norm{f}_{\Phi,Q}\leq 3^n\norm{f}_{\Phi,Q_0}\leq 3^n M_{\Phi, \mathcal{D}^{(i)}}(f)(x)\leq 3^n\s{i=1}{3^n}M_{\Phi, \mathcal{D}^{(i)}}(f)(x).\]
Thus, by taking supremum over all cubes $Q$ that contain $x$ we have the desired estimate. From \eqref{control diadico de M phi}, it will be sufficient to prove Theorem~\ref{teorema principal} for $M_{\Phi,\mathcal{D}}$, for a general dyadic grid $\mathcal{D}$.


\section{Proof of the main result}

We devote this section to proving Theorem~\ref{teorema principal}. We shall split some parts into several claims that will be proved separately for the sake of simplicity. 

First, we shall give some lemmas that will be useful in the proof of our main result.

\begin{lema}\label{descomposicion_de_CZ_del_espacio}
	Given $\lambda>0$, a bounded function with compact support  $f$, a dyadic grid $\mathcal{D}$ and a Young function $\varphi$, there exists a family of dyadic cubes $\{Q_j\}$ of $\mathcal{D}$ that satisfies
	\[\{x\in\R^n: M_{\varphi,\mathcal{D}}f(x)>\lambda\}=\bigcup_j Q_j,\]
	and $\norm{f}_{\varphi,Q_j}>\lambda$  for every $j$.
\end{lema}

A proof of this lemma can be found in \cite[Lemma 5]{Berra}. Notice that the cubes $Q_j$ are maximal in the sense of inclusion, that is, if $Q_j\subsetneq Q'$ for a fixed $j$, then $\norm{f}_{\Phi,Q'}\leq \lambda$.

\medskip

\begin{lema}\label{estimacion constante epsilon}
	Let $f$ be the function defined in $[0,\infty)$ by 
	\[f(x)=\left\{\begin{array}{rcl}
	\left(1+\frac{1}{x}\right)^{\frac{x}{1+x}},& \textrm{ if } & x>0;\\
	1, & \textrm{ if } & x=0.
	\end{array}
	\right.\]
	Then  $1\leq f(x)\leq e^{1/e}$, for every $x\geq 0$.
\end{lema}

\begin{proof}
	Observe that
	\[f'(x)=f(x)\frac{1}{(1+x)^2}\left(\log\left(1+\frac{1}{x}\right)-1\right) \quad \textrm{if } \quad x>0.\]
	It is easy to see that $f$ has a local maximum at $x=1/(e-1)$ and  $f(1/(e-1))=e^{1/e}$. On the other hand,  
	\[\lim_{x\to 0^+}f(x)=\lim_{x\to\infty}f(x)=1,\]  
	which directly implies the thesis.
\end{proof}

\begin{proof}[Proof of Theorem~\ref{teorema principal}]
	
	Fix $t>0$, a dyadic grid $\mathcal{D}$ and denote $g=fv/t$. Then, it will be enough to prove that
	\[uv^r\left(\left\{x\in \R^n: M_{\Phi,\mathcal{D}}(g)(x)>v(x)\right\}\right)\leq C\int_{\R^n}\Phi\left(\frac{f}{t}\right)\,uv^r\,dx.\]
	We can assume, without loss of generality, that $g$ is a bounded function with compact support. Fix a number $a>2^n$ and, for every $k\in\Z$ we will define the set
	\[\Omega_k=\{x\in\R^n: M_{\mathcal{D}}v(x)>a^k\}\cap\{x\in\R^n: M_{\Phi,\mathcal{D}}g(x)>a^k\},\]
	which can be written as a disjoint union of maximal dyadic cubes $\{Q^k_j\}_j$, for every $k$, by virtue of Lemma~\ref{descomposicion_de_CZ_del_espacio}. 
	
	Let us now consider the set $\Gamma=\{(k,j): |Q^k_j\cap\{x: v(x)\leq a^{k+1}\}|>0\}$. Therefore, since $v\in A_1$, for $(k,j)\in\Gamma$ we have that
	\begin{equation}\label{propiedad_cubos_para_v}
	\frac{a^k}{[v]_{A_1}}\leq \frac{1}{[v]_{A_1}}\inf_{Q^k_j}M_\mathcal{D}v\leq \inf_{Q^k_j}v\leq \frac{1}{|Q^k_j|}\int_{Q^k_j}v\leq [v]_{A_1}\inf_{Q^k_j}v\leq[v]_{A_1}a^{k+1}.
	\end{equation}
	
	Since $v^r\in A_1$, we also have
	\[\frac{1}{|Q^k_j|}\int_{Q^k_j}v\leq  \left(\frac{1}{|Q^k_j|}\int_{Q^k_j}v^r\right)^{1/r}\leq [v^r]_{A_1}^{1/r}\frac{1}{|Q^k_j|}\int_{Q^k_j}v.\]
	By combining the estimate above with \eqref{propiedad_cubos_para_v} we get
	\begin{equation}\label{propiedad_cubos_para_v^r}
	\frac{a^{kr}}{[v]_{A_1}^r}\leq\frac{1}{|Q^k_j|}\int_{Q^k_j}v^r\leq [v^r]_{A_1}[v]_{A_1}^r a^r a^{kr}.
	\end{equation}
	
	Now observe that if we set $A_k=\left\{x: a^k<v(x)\leq a^{k+1}\right\}$, then for every $k$ we have that
	\begin{align*}
	A_k\cap\{ M_{\Phi,\mathcal{D}}g>v\}&\subseteq\left\{ M_\mathcal{D}v>a^k\right\}\cap\left\{ v\leq a^{k+1}\right\}\cap\left\{ M_{\Phi,\mathcal{D}}g>a^k\right\}\\
	&\subseteq\bigcup_{j: (k,j)\in \Gamma}Q^k_j,
	\end{align*}
	except for a set of null measure. Thus,
	\begin{align*}
	uv^r(\{x\in\R^n: M_{\Phi,\mathcal{D}}g>v\})&=\s{k\in\Z}{}uv^r(\{M_{\Phi,\mathcal{D}}g>v\}\cap A_k)\\
	&\leq a^r\s{k\in\Z}{}a^{kr}u(\{M_{\Phi,\mathcal{D}}g>v\}\cap A_k)\\
	&\leq a^r\s{k\in\Z}{}\s{j:(k,j)\in\Gamma}{}a^{kr}u(Q^k_j)\\
	&\leq a^r[v]_{A_1}^r\s{k\in\Z}{}\s{j:(k,j)\in\Gamma}{}u(Q^k_j)\frac{v^r(Q^k_j)}{|Q^k_j|},
	\end{align*}
	where we have used \eqref{propiedad_cubos_para_v^r}.
	
	Fix now a negative integer $N$ and define $\Gamma_N=\{(k,j)\in \Gamma: k\geq N\}$. The objective is to prove that there exists a positive constant $C$, independent of $N$ for which
	\[\s{k\geq N}{}{}\s{j:(k,j)\in\Gamma_N}{}u(Q^k_j)\frac{v^r(Q^k_j)}{|Q^k_j|}\leq C\int_{\R^n}\Phi\left(\frac{f}{t}\right)uv^r\,dx.\]
	If we can accomplish this estimate, the result will follow by letting $N\to-\infty$.
	
	\medskip
	
	Let $\Delta_N=\{Q^k_j: (k,j)\in \Gamma_N\}$. Given two cubes in $\Delta_N$ either they are disjoint, or one is contained in the other. Also observe that, if $k>t$,
	$\Omega_k\subseteq \Omega_t$. Thus, if the cubes $Q^k_j$ and $Q_s^t$ verify $Q^k_j\cap Q_s^t\neq\emptyset$, then necessarily we must have  $Q^k_j\subseteq Q_s^t$.
	
	Since $v^r\in A_1\subset A_{\infty}$, there exist positive constants $[v^r]_{A_\infty}$ and $\eta$ such that, for every cube $Q\subset \R^n$ and $E$ a measurable subset of $Q$
	\[\frac{v^r(E)}{v^r(Q)}\leq [v^r]_{A_\infty}\left(\frac{|E|}{|Q|}\right)^\eta.\] 
	
	Let $0<\beta<\eta$ and define inductively a sequence of sets as follows: 
	\[G_0=\{(k,j)\in \Gamma_N: Q^k_j \textrm{ is maximal in }\Delta_N\},\]
	and, in a colloquial way, a pair $(k,j)$ in $\Gamma_N$ belongs to $G_{n+1}$ if the cube $Q^k_j$ has an ``ancestor''  $Q_s^t$, with $(t,s)\in G_n$, and
	$Q^k_j$ is the ``first descendant'' in $\Gamma_N$ that satisfies $\mu(Q_j^k)>\mu(Q_s^t)$, in the sense that $\mu(Q_i^\ell)\leq\mu(Q_s^t)$ for each $(\ell,i)\in\Gamma_N$ and $Q^k_j\subsetneq Q_i^\ell\subseteq Q_s^t$, where  $\mu(Q_s^t)$ is the modified average
	\[\mu(Q_s^t):= \frac{1}{a^{\beta  t}}\frac{1}{|Q_s^t|}\int_{Q_s^t}u(x)\,dx.\]
	 That is, we define for $n\geq 0$, $G_{n+1}$ as the set of pairs $(k,j)\in \Gamma_N$ for which there exists $(t,s)\in G_n$ with $Q_j^k\subsetneq Q_s^t$ and the inequalities
	
	\begin{equation}\label{desigualdad1_conjuntoGn}
	\frac{1}{|Q^k_j|}\int_{Q^k_j}u(x)\,dx>a^{(k-t)\beta } \frac{1}{|Q_s^t|}\int_{Q_s^t}u(x)\,dx,
	\end{equation}
	and
	\begin{equation}\label{desigualdad2_conjuntoGn}
	\frac{1}{|Q_i^\ell|}\int_{Q_i^\ell}u(x)\,dx\leq a^{(\ell-t)\beta } \frac{1}{|Q_s^t|}\int_{Q_s^t}u(x)\,dx
	\end{equation}
	hold with $(\ell,i)\in\Gamma_N$ and $Q^k_j\subsetneq Q_i^\ell\subseteq Q_s^t$.

	Observe that, if $G_{n_0}=\emptyset$ for some $n_0$, then $G_{n}=\emptyset$ for every $n\geq n_0$.
	Let $P=\bigcup_{n\geq 0}G_n$. If $(t,s)\in P$ we will say that $Q_s^t$ is a \textit{principal cube}.
	
	We now state some claims whose proofs will be given at the end of this section.
	\begin{afirmacion}\label{afirmacion1}
		There exists a positive constant $C$ such that
		\[\s{(k,j)\in\Gamma_N}{}\frac{v^r(Q^k_j)}{|Q^k_j|}u(Q^k_j)\leq C \s{(k,j)\in P}{}\frac{v^r(Q^k_j)}{|Q^k_j|}u(Q^k_j).\]
	\end{afirmacion}
	
	For every fixed $k\in\Z$, we consider the family $\{\tilde Q_i^k\}_i$ of maximal dyadic cubes given by Lemma~\ref{descomposicion_de_CZ_del_espacio}, which decompose the set $\{x\in \R^n: M_{\Phi,\mathcal{D}}g(x)>a^k\}$.  Then, for every $i$, it follows that
	\begin{equation}
	a^k<\norm{g}_{\Phi,\tilde Q_i^k} \quad\textrm{ or, equivalently }\quad 1<\norm{\frac{g}{a^k}}_{\Phi,\tilde Q_i^k}.
	\end{equation}
	
	\begin{afirmacion}\label{afirmacion2}
		There exists a positive constant $C$ such that
		\begin{equation}\label{acotacion akr por promedio}
		a^{kr}\leq C\frac{1}{|\tilde Q_i^k|}\int_{\tilde Q_i^k}\Phi\left(\frac{f}{t}\right)v^r(x)\,dx,
		\end{equation}
		 for every cube $\tilde Q_i^k$.
		\end{afirmacion}

	Since $Q^k_j\subseteq\{x: M_{\Phi,\mathcal{D}}g(x)>a^k\}$, for each $j$ there exists a unique $i=i(j,k)$ for which $Q^k_j\subseteq \tilde Q_i^k$. By applying Claims~\ref{afirmacion1} and ~\ref{afirmacion2} we have that 
	\begin{align*}
	\s{(k,j)\in \Gamma_N}{}\frac{1}{|Q^k_j|}v^r(Q^k_j)u(Q^k_j)&\leq C\s{(k,j)\in P}{}\frac{1}{|Q^k_j|}v^r(Q^k_j)u(Q^k_j)\\
	&\leq C \s{(k,j)\in P}{}a^{kr}u(Q^k_j)\\
	&\leq C\s{(k,j)\in P}{}\frac{u(Q^k_j)}{|\tilde Q_i^k|}\int_{\tilde Q_i^k}\Phi\left(\frac{f}{t}\right)v^r(x)\,dx\\	
	&=C\int_{\R^n}\Phi\left(\frac{f}{t}\right)v^r(x)\left[\s{(k,j)\in P}{}\frac{1}{|\tilde Q_i^k|}u(Q^k_j)\mathcal{X}_{\tilde Q_i^k}(x)\right]\,dx\\
	&=C\int_{\R^n}\Phi\left(\frac{f}{t}\right)h(x)v^r(x)\,dx,
	\end{align*}
	where $h(x)=\s{(k,j)\in P}{}|\tilde Q_i^k|^{-1}u(Q^k_j)\mathcal{X}_{\tilde Q_i^k}(x)$.

	In order to finish, it only remains to show that there exists a positive constant $C$ such that $h(x)\leq Cu(x)$. The proof follows similar lines as in \cite{Sawyer}. We include it for the sake of completeness.
	
	Indeed, given $x\in \R^n$, we can assume that $u(x)<\infty$. For every fixed $k$ there exists, at most, one $\tilde Q_i^k$ which satisfies $x\in \tilde Q_i^k$. If this cube does exist, we denote it by $\tilde Q^k$ and for every $k$ we define $P_k=\{(k,j)\in P: Q^k_j\subseteq \tilde Q^k\}$ and $G=\{k: P_k\neq\emptyset\}$. Recall that $k\geq N$, so $G$ is bounded from below. Let $k_0$ be the minimum of $G$. We shall build a sequence in $G$ in the following way: chosen $k_m$, for $m\geq 0$ we select $k_{m+1}$ as the smallest integer in $G$, greater than $k_m$ and verifying
	\begin{equation}\label{desigualdad1_sucesionkm}
	\frac{1}{|\tilde Q^{k_{m+1}}|}\int_{\tilde Q^{k_{m+1}}}u(y)\,dy>\frac{2}{|\tilde Q^{k_m}|}\int_{\tilde Q^{k_m}}u(y)\,dy.
	\end{equation}
	It is clear that, if $\ell\in G$ and $k_m\leq \ell< k_{m+1}$, then
	\begin{equation}\label{desigualdad2_sucesionkm}
	\frac{1}{|\tilde Q^{\ell}|}\int_{\tilde Q^{\ell}}u(y)\,dy\leq\frac{2}{|\tilde Q^{k_m}|}\int_{\tilde Q^{k_m}}u(y)\,dy.
	\end{equation}
	
	The so-defined sequence $\{k_m\}_{m\geq0}$ has only a finite number of terms. Indeed, if it was not the case, by applying condition \eqref{desigualdad1_sucesionkm} repeatedly,  we would have
	\[[u]_{A_1}u(x)\geq \frac{1}{|\tilde Q^{k_m}|}\int_{\tilde Q^{k_m}}u(y)\,dy>2^m\frac{1}{|\tilde Q^{k_0}|}\int_{\tilde Q^{k_0}}u(y)\,dy\]
	for every $m\in \N$, and by letting $m\to\infty$ we would get a contradiction. Thus $\{k_m\}=\{k_m\}_{m=0}^{m_0}$.
	
	By denoting $F_m=\{\ell\in G: k_m\leq \ell<k_{m+1}\}$ we can write
	
	\begin{align*}
	h(x)&=\s{(k,j)\in P}{}\frac{1}{|\tilde Q_i^k|}u(Q^k_j)\mathcal{X}_{\tilde Q_i^k}(x)\\
	&=\s{(k,j)\in P}{}\frac{u(Q^k_j)}{u(\tilde Q^k)}\left(\frac{1}{|\tilde Q^k|}\int_{\tilde Q^k}u(y)\,dy\right)\\
	&=\s{m=0}{m_0}\s{\ell\in F_m}{}\left(\frac{1}{|\tilde Q^\ell|}\int_{\tilde Q^\ell}u(y)\,dy\right)\s{j:(\ell,j)\in P_\ell}{}\frac{u(Q_j^\ell)}{u(\tilde Q^\ell)}\\
	&\leq 2\s{m=0}{m_0}\left(\frac{1}{|\tilde Q^{k_m}|}\int_{\tilde Q^{k_m}}u(y)\,dy\right)\s{\ell\in F_m}{}\s{j: (\ell,j)\in P_\ell}{}\frac{u(Q_j^\ell)}{u(\tilde Q^\ell)},
	\end{align*}
    where in the last inequality we have used \eqref{desigualdad2_sucesionkm}.
	
	\begin{afirmacion}\label{afirmacion3}
		There exists a positive constant $C$ such that
		\[\s{\ell\in F_m}{}\s{j:(\ell,j)\in P_\ell}{}\frac{u(Q_j^\ell)}{u(\tilde Q^\ell)}\leq C.\]
	\end{afirmacion}
	
	If this claim holds, we are done. Indeed, denoting with $C_m=|\tilde Q^{k_m}|^{-1}\int_{\tilde Q^{k_m}}u$, using the estimation above, and \eqref{desigualdad1_sucesionkm} we have that
	\begin{align*}
	h(x)&\leq C\s{m=0}{m_0}C_m\leq C\s{m=0}{m_0} C_{m_0}2^{m-m_0}\\
	&= CC_{m_0}2^{-m_0}\s{m=0}{m_0}2^m=CC_{m_0}2^{-m_0}(2^{m_0+1}-1)\\
	&\leq CC_{m_0}\leq C[u]_{A_1}u(x).\qedhere
	\end{align*}
\end{proof}

In order to conclude, we will prove the claims.

\begin{proof}[Proof of Claim~\ref{afirmacion1}]
	Fix $(t,s)\in P$ and define
	\[
	I(t,s)=\left\{(k,j)\in \Gamma_N: Q^k_j\subseteq Q_s^t \textrm{ and $Q_s^t$ is the smallest principal cube that contains } Q^k_j\right\}.\]
	Particularly, every $Q^k_j$ with $(k,j)\in I(t,s)$ is not principal, unless $(k,j)=(t,s)$. By condition~\eqref{desigualdad2_conjuntoGn} we can write
	\begin{align*}
	\s{(k,j)\in I(t,s)}{}\frac{v^r(Q^k_j)}{|Q^k_j|}u(Q^k_j)&\leq \s{(k,j)\in I(t,s)}{}a^{(k-t)\beta }\frac{u(Q_s^t)}{|Q_s^t|}v^r(Q^k_j)\\
	&\leq C\frac{u(Q_s^t)}{|Q_s^t|}\s{(k,j)\in I(t,s)}{}a^{(k-t)\beta }v^r(Q^k_j)\\
	&\leq C\frac{u(Q_s^t)}{|Q_s^t|}\s{(k,j)\in I(t,s)}{}a^{(k-t)\beta }v^r(Q_s^t\cap\{x: M_{\mathcal{D}}v(x)>a^k\}).
	\end{align*}
	On the other hand, by the $A_\infty$-condition of $v^r$ and \eqref{propiedad_cubos_para_v} we obtain that
	\begin{align*} 
	v^r(Q_s^t\cap\{x: M_{\mathcal{D}}v(x)>a^k\})&=\frac{v^r(Q_s^t\cap\{x: M_{\mathcal{D}}v(x)>a^k\})}{v^r(Q_s^t)}v^r(Q_s^t)\\
	&\leq [v^r]_{A_\infty}\left(\frac{|Q_s^t\cap\{x: M_{\mathcal{D}}v(x)>a^k\}|}{|Q_s^t|}\right)^{\eta}v^r(Q_s^t)\\
	&\leq [v^r]_{A_\infty}\left(\frac{[v]_{A_1}}{|Q_s^t|a^k}\int_{Q_s^t}v(x)\,dx\right)^{\eta}v^r(Q_s^t)\\
	&\leq [v^r]_{A_\infty}(a[v]_{A_1}^2)^{\eta}a^{(t-k)\eta}v^r(Q_s^t).
	\end{align*}
	
	Combining these two estimates, we have that
	
	\begin{align*}
	\s{(k,j)\in I(t,s)}{}\frac{v^r(Q^k_j)}{|Q^k_j|}u(Q^k_j)&\leq C\frac{u(Q_s^t)}{|Q_s^t|}\s{(k,j)\in I(t,s)}{}a^{(k-t)\beta }a^{(t-k)\eta}v^r(Q_s^t)\\
	&= C\frac{u(Q_s^t)}{|Q_s^t|}v^r(Q_s^t)\s{k\geq t}{}a^{(t-k)(\eta-\beta)}\\
	&\leq C\frac{u(Q_s^t)}{|Q_s^t|}v^r(Q_s^t),
	\end{align*}
	since $\eta-\beta>0$ and $a>2^n>1$. So, we have obtained that
	\[\s{(k,j)\in I(t,s)}{}\frac{v^r(Q^k_j)}{|Q^k_j|}u(Q^k_j)\leq C\frac{u(Q_s^t)}{|Q_s^t|}v^r(Q_s^t),\]
	and if we sum over all $(t,s)\in P$ it follows that
	\[\s{(k,j)\in \Gamma_N}{}\frac{v^r(Q^k_j)}{|Q^k_j|}u(Q^k_j)\leq\s{(t,s)\in P}{}\s{(k,j)\in I(t,s)}{}\frac{v^r(Q^k_j)}{|Q^k_j|}u(Q^k_j)\leq C\s{(t,s)\in P}{}\frac{v^r(Q_s^t)}{|Q_s^t|}u(Q_s^t),\]
	since $\bigcup_{(t,s)\in P}I(t,s)=\Gamma_N$.
\end{proof}

\bigskip

\begin{proof}[Proof of Claim~\ref{afirmacion2}]

	Fix one of these cubes $\tilde Q_i^k$. We define the sets $A=\{x\in \tilde Q_i^k: v(x)\leq t_0a^k\}$, where $t_0$ is given in \eqref{propiedad acotacion de Phi} and $B=\tilde Q_i^k\backslash A$. Thus,
	\[1<\norm{\frac{g}{a^k}}_{\Phi,\tilde Q_i^k}\leq \norm{\frac{g}{a^k}\mathcal{X}_A}_{\Phi,\tilde Q_i^k}+\norm{\frac{g}{a^k}\mathcal{X}_B}_{\Phi,\tilde Q_i^k}=I+II.\]
	
	Then we can deduce that either $I>1/2$ or $II>1/2$. If the first case holds, from the submultiplicativity and the lower type of $\Phi$ we have
	\begin{align*}
	1&<\frac{1}{|\tilde Q_i^k|}\int_A \Phi\left(\frac{2fv}{ta^k}\right)\,dx\\
	&\leq \Phi(2t_0)\frac{1}{|\tilde Q_i^k|}\int_A \Phi\left(\frac{f}{t}\right)\Phi\left(\frac{v}{t_0a^k}\right)\,dx\\
	&\leq C_r\Phi(2t_0)\frac{1}{|\tilde Q_i^k|}\int_{\tilde Q_i^k} \Phi\left(\frac{f}{t}\right)\frac{v^r}{t_0^ra^{kr}}\,dx\\
	&= C_r\frac{\Phi(2t_0)}{t_0^ra^{kr}}\frac{1}{|\tilde Q_i^k|}\int_{\tilde Q_i^k} \Phi\left(\frac{f}{t}\right)v^r\,dx,
	\end{align*}
	which means that
	\begin{equation*}
	a^{kr}<C_r\frac{\Phi(2t_0)}{t_0^r}\frac{1}{|\tilde Q_i^k|}\int_{\tilde Q_i^k} \Phi\left(\frac{f}{t}\right)v^r\,dx.
	\end{equation*}
	
	For the second case, if we set $J_{i,k}=\{\ell: Q_\ell^k\subset \tilde Q_i^k\}$, then $B\subset \bigcup_{\ell\in J_{i,k}}Q_\ell^k$. Indeed, $B=\tilde Q_i^k\cap\{v>t_0a^k\}\subset \tilde Q_i^k\cap\{M_\mathcal{D} v>a^k\}$, and this implies that $B\subset \tilde Q_i^k\cap\{M_{\mathcal{D}}v>a^k\}\cap\{x: M_{\Phi,\mathcal{D}}g(x)>a^k\}=\tilde Q_i^k\cap \Omega_k$. Hence, 
	\begin{equation}\label{descomposicion de B en cubos}
	B\subset \tilde Q_i^k\cap \left(\bigcup_{\ell\in J_{i,k}} Q_\ell^k\right)=\bigcup_{\ell\in J_{i,k}} \left(\tilde Q_i^k\cap Q_\ell^k\right)=\bigcup_{\ell\in J_{i,k}} Q_\ell^k.
	\end{equation}
	Then, since $\Phi\in \mathcal{F}_r$ we have that
	\begin{align*}
	1&<\frac{1}{|\tilde Q_i^k|}\int_B \Phi\left(\frac{2fv}{ta^k}\right)\,dx\\
	&\leq \Phi(2)\frac{1}{|\tilde Q_i^k|}\int_B \Phi\left(\frac{f}{t}\right)\Phi\left(\frac{v}{a^k}\right)\left(\frac{v}{a^k}\right)^{-r}\frac{v^r}{a^{kr}}\mathcal{X}_{\tilde Q_i^k\cap B}\,dx\\
	&\leq \Phi(2)C_0\frac{1}{|\tilde Q_i^k|}\int_B \Phi\left(\frac{f}{t}\right)\frac{v^r}{a^{kr}}\left(\log\left(\frac{v}{a^k}\right)\right)^\delta\mathcal{X}_{\tilde Q_i^k\cap B}\,dx,
	\end{align*}
	and using \eqref{descomposicion de B en cubos} we can write
	\begin{align*}
	a^{kr}&< C\frac{1}{|\tilde Q_i^k|}\int_B \Phi\left(\frac{f}{t}\right)v^r\left(\log\left(\frac{v}{a^k}\right)\right)^\delta\mathcal{X}_{\tilde Q_i^k\cap B}\,dx\\
	&\leq \frac{C}{|\tilde Q_i^k|}\s{\ell\in J_{i,k}}{}\int_{Q_\ell^k} \Phi\left(\frac{f}{t}\right)w_kv^r\,dx,
	\end{align*}
	where $w_k=\left(\log\left(\frac{v}{a^k}\right)\right)^\delta\mathcal{X}_{Q_\ell^k\cap B}$.
	\medskip
	
	Note that, since $v^r\in A_1$, there exists an exponent $s>1$ such that $v^r\in \mathrm{RH}_s$. Let 
	\[\delta_0=\max\{\delta,1\}\quad\textrm{ and }\quad \e_0=\min\{1/(s'[v]_{A_1}^{1/s'}a^{1/s'}[v^r]_{RH_s}\delta_0-1),1\}.\]
	Fix $0<\e\leq \e_0$ and $\gamma=1+\e$, so $\gamma'=1+1/\e$. Thus, by applying H\"{o}lder's inequality with exponents $\gamma$ and $\gamma'$,  with respect to the measure $d\mu=v^r(x)\,dx$, we obtain
	\begin{equation}\label{estimacion parte B 1 caso r general}
	a^{kr}<\frac{C_0\Phi(2)}{|\tilde Q_i^k|}\s{\ell\in J_{i,k}}{}v^r(Q_\ell^k)\left(\frac{1}{v^r(Q_\ell^k)}\int_{Q_\ell^k}\Phi^\gamma\left(\frac{f}{t}\right)v^r\,dx\right)^{1/\gamma}\left(\frac{1}{v^r(Q_\ell^k)}\int_{Q_\ell^k}w_k^{\gamma'}v^r\,dx\right)^{1/\gamma'}.
	\end{equation}
	
	Next, we prove that the second average is bounded by a positive constant $K$, independent of $\e$. Indeed, by using \eqref{propiedad_cubos_para_v}, the fact that $\log(t)\lesssim \xi^{-1}t^{\xi}$, for every $t,\xi>0$ and H\"{o}lder's inequality with exponents $s$ and $s'$ we have that 
	
	\begin{align*}
	\left(\frac{1}{v^r(Q_\ell^k)}\int_{Q_\ell^k}  w_k^{\gamma'}v^r(x)\,dx\right)^{1/\gamma'}&= \left(\frac{1}{v^r(Q_\ell^k)}\int_{Q_\ell^k\cap B} \log\left(\frac{v(x)}{a^k}\right)^{\delta\gamma'}v^r(x)\,dx\right)^{1/\gamma'}\\
	&\leq \left(\frac{1}{v^r(Q_\ell^k)}\int_{Q_\ell^k\cap B} \delta s'\gamma'\left(\frac{v(x)}{a^k}\right)^{1/s'}v^r(x)\,dx\right)^{1/\gamma'}\\
	&\leq  \left(\delta s'\gamma'\frac{|Q_\ell^k|}{v^r(Q_\ell^k)}\left(\frac{1}{|Q_\ell^k|}\int_{Q_\ell^k}\frac{v(x)}{a^k}\right)^{1/s'}\left(\frac{1}{|Q_\ell^k|}\int_{Q_\ell^k}v^{rs}(x)\,dx\right)^{1/s}\right)^{1/\gamma'}\\
	&\leq \left([v]_{A_1}^{1/s'}a^{1/s'}[v^r]_{\mathrm{RH}_s} \delta_0\gamma's'\right)^{1/\gamma'}\\
	&\leq (\gamma'\gamma')^{1/\gamma'}\\
	&\leq e^{2/e},
	\end{align*}
	where we have used the definition of $\e$ and Lemma~\ref{estimacion constante epsilon}. Thus, we can choose $K=e^{2/e}$. Then, we have proved that
	\begin{equation}\label{estimacion akr 1}
	a^{kr}< \frac{C}{|\tilde Q_i^k|}\s{\ell \in J_{i,k}}{}v^r(Q_\ell^k)\left(\frac{1}{v^r(Q_\ell^k)}\int_{Q_\ell^k}\Phi^\gamma\left(\frac{f}{t}\right)v^r\,dx\right)^{1/\gamma},
	\end{equation}
	and observe that if we set $\Psi(t)=t^{\gamma}$, the expression between brackets is $\norm{\Phi(f/t)}_{\Psi,v^r,Q_\ell^k}$. By using \eqref{equivalencia normal Luxemburo con infimo} we have that
	\begin{equation}\label{acotacion norma luxemburgo de Phi}
	\norm{\Phi(f/t)}_{\Psi,v^r,Q_\ell^k}\leq \tau+\frac{\tau^{1-\gamma}}{v^r(Q_\ell^k)}\int_{Q_\ell^k}\Phi^\gamma\left(\frac{f}{t}\right)v^r\,dx,
	\end{equation}
	for every $\tau>0$.
	By combining estimates \eqref{estimacion akr 1} and \eqref{acotacion norma luxemburgo de Phi} we have
	\begin{equation}\label{estimacion akr 2}
	a^{kr}<C\tau\frac{v^r(\tilde Q_i^k)}{|\tilde Q_i^k|}+\frac{C\tau^{1-\gamma}}{|\tilde Q_i^k|}\int_{\tilde Q_i^k}\Phi^\gamma\left(\frac{f}{t}\right)v^r\,dx.
	\end{equation}
	Observe that
	    
	    \begin{equation}\label{estimacion promedio cubos grandes}
	    \frac{1}{|\tilde Q_i^k|}\int_{\tilde Q_i^k}v^r(x)\,dx\leq (1+a^r[v]_{A_1}^r[v^r]_{A_1})a^{kr}.
		\end{equation}
		Indeed, notice that
		\[\{x\in \tilde Q_i^k: v(x)>a^k\}\subset \bigcup_{\ell\in J_{i,k}} Q_\ell^k,\]
		and then 
		\begin{align*}
		\frac{1}{|\tilde Q_i^k|}\int_{\tilde Q_i^k} v^r\,dx&=\frac{1}{|\tilde Q_i^k|}\left[\int_{\tilde Q_i^k\cap\{v\leq a^k\}}v^r\,dx+\int_{\tilde Q_i^k\cap\{v> a^k\}}v^r\,dx\right]\\
		&\leq a^{kr}+\s{j\in J_{i,k}}{}{\frac{|Q_\ell^k|}{|\tilde Q_i^k|}\frac{1}{|Q_\ell^k|}\int_{Q_\ell^k}v^r\,dx}\\
		&\leq a^{kr}+[v]_{A_1}^r[v^r]_{A_1}a^ra^{kr}\\
		&\leq(1+[v]_{A_1}^r[v^r]_{A_1}a^r)a^{kr},
		\end{align*}
		by virtue of \eqref{propiedad_cubos_para_v^r}.
		
		Thus, we select $\tau$ such that $C\tau(1+[v]_{A_1}^r[v^r]_{A_1}a^r)=1/2$. From  \eqref{estimacion akr 2},
		\begin{align*}
		a^{kr}&<2C\left(\frac{1}{2C(1+[v]_{A_1}^r[v^r]_{A_1}a^r)}\right)^{1-\gamma}\frac{1}{|\tilde Q_i^k|}\int_{\tilde Q_i^k}\Phi^\gamma\left(\frac{f}{t}\right)v^r\,dx\\
		&\leq C\frac{1}{|\tilde Q_i^k|}\int_{\tilde Q_i^k}\Phi^\gamma\left(\frac{f}{t}\right)v^r\,dx,
		\end{align*}
		for every $0<\e\leq \e_0$. Then by letting $\e\to 0$ and using the Dominate Convergence Theorem, we get
		\[a^{kr}<(4C)^2(1+[v]_{A_1}^r[v^r]_{A_1}a^r)\frac{1}{|\tilde Q_i^k|}\int_{\tilde Q_i^k}\Phi\left(\frac{f}{t}\right)v^r\,dx,\]
		which completes the proof of the claim. 
	\end{proof}	

\bigskip

\begin{proof}[Proof of Claim~\ref{afirmacion3}]
	Let us first assume that, if $(\ell,j)\in P_\ell$ and $k_m\leq \ell<k_{m+1}$, then
	\begin{equation}\label{ec1_demo_afirmacion3}
	\frac{1}{|Q_j^\ell|}\int_{Q_j^\ell}u(y)\,dy>\frac{a^{(\ell-k_m)\beta}}{2[u]_{A_1}}\frac{1}{|\tilde Q^\ell|}\int_{\tilde Q^\ell}u(y)\,dy.
	\end{equation}
	Thus, if $y\in Q_j^\ell$,
	\begin{align*}
	u(y)[u]_{A_1}&\geq\frac{1}{|Q_j^\ell|}\int_{Q_j^\ell}u(z)\,dz\\
	&>\frac{a^{(\ell-k_m)\beta}}{2[u]_{A_1}}\frac{1}{|\tilde Q^\ell|}\int_{\tilde Q^\ell}u(y)\,dy
	\end{align*}
	and consequently
	\[u(y)>\frac{a^{(\ell-k_m)\beta}}{2[u]_{A_1}^2}\frac{u(\tilde Q^\ell)}{|\tilde Q^\ell|}=:\lambda,\]
	which implies that
	\[\bigcup_{j:(\ell,j)\in P_\ell}Q_j^\ell\subseteq\{x\in \tilde Q^\ell: u(x)>\lambda\}.\]
	Since $u\in A_1\subseteq A_{\infty}$, there exist positive constants $[u]_{A_\infty}$ and $\nu$ for which $\frac{u(E)}{u(Q)}\leq [u]_{A_\infty}\left(\frac{|E|}{|Q|}\right)^{\nu}$ holds, for every measurable set $E\subseteq Q$. Then, from Chebyshev's inequality and the definition of $\lambda$ we have that
	\begin{align*}
	\s{j:(\ell,j)\in P_\ell}{}u(Q_j^\ell)&=u\left(\bigcup_{j:(\ell,j)\in P_\ell} Q_j^\ell\right)\\
	&\leq u(\{x\in \tilde Q^\ell: u(x)>\lambda\})\\
	&\leq Cu(\tilde Q^\ell)\left(\frac{|\{x\in \tilde Q^\ell: u(x)>\lambda\}|}{|\tilde Q^\ell|}\right)^{\nu}\\
	&\leq Cu(\tilde Q^\ell)\left(\frac{1}{\lambda|\tilde Q^\ell|}\int_{\tilde Q^\ell}u(y)\,dy\right)^{\nu}\\
	&=Cu(\tilde Q^\ell)a^{(k_m-\ell)\beta\nu},
	\end{align*}
	and finally
	\begin{align*}
	\s{\ell\in F_m}{}\s{j:(\ell,j)\in P_\ell}{}\frac{u(Q_j^\ell)}{u(\tilde Q^\ell)}&\leq C\s{\ell\in F_m}{}a^{(k_m-\ell)\beta\nu}\\
	&\leq C\s{\ell\geq k_m}{}a^{(k_m-\ell)\beta\nu}=C,
	\end{align*}
	since $a>1$. This completes the proof.
	
	\bigskip
	
	In order to finish we will prove that ~\eqref{ec1_demo_afirmacion3} holds. Select $(\ell,j)\in P_\ell$ with $k_m\leq \ell<k_{m+1}$. Since $\Omega_\ell\subseteq \Omega_{k_m}$, by maximality, there exists a unique $s$ verifying $Q_j^\ell\subseteq Q_s^{k_m}$. We shall see that $(k_m,s)\in \Gamma_N$. If $(k_m,s)\in P$ is trivial because $P\subseteq \Gamma_N$. Then, assume that $(k_m,s)\not\in P$. From the definition of $G$ and $P_{k_m}$, $\tilde Q^{k_m}$ contains a cube $Q_p^{k_m}$ with $(k_m,p)\in P$. We shall see, as a first step, that $Q_s^{k_m}\subsetneq \tilde Q^{k_m}$. Indeed, there exists a unique $i(s)$ for which $Q_j^\ell\subseteq Q_s^{k_m}\subseteq \tilde Q_{i(s)}^{k_m}$. Besides,
	\[\left\{x: M_{\Phi,\mathcal{D}}g(x)>a^\ell\right\}\subseteq \left\{x: M_{\Phi,\mathcal{D}}g(x)>a^{k_m}\right\}=\bigcup_i \tilde Q_i^{k_m},\]
	and this implies that there exists a unique $i_0$ such that $Q_j^\ell \subseteq \tilde Q^\ell\subseteq \tilde Q_{i_0}^{k_m}$. On the other hand, from the definition of $\tilde Q^k$, $x\in \tilde Q^{k_m}$ and $x\in \tilde Q_{i_0}^{k_m}$, and we must have
	\[\tilde Q_{i(s)}^{k_m}=\tilde Q_{i_0}^{k_m}=\tilde Q^{k_m},\]
	which directly implies that $Q_s^{k_m}\subseteq \tilde Q^{k_m}$. In fact, this inclusion is proper, because both $Q_s^{k_m}$ and $Q_p^{k_m}$ are contained in $\tilde Q^{k_m}$, and also $s\neq p$.
	
	Observe that $\tilde Q^{k_m}$ is a maximal cube of the set $\{x: M_{\Phi,\mathcal{D}}g(x)>a^{k_m}\}$ and $Q_s^{k_m}$ is a maximal cube of
	\[\Omega_{k_m}=\{x\in\R^n: M_{\mathcal{D}}v(x)>a^{k_m}\}\cap\{x\in\R^n: M_{\Phi,\mathcal{D}}g(x)>a^{k_m}\}.\]
	Since $Q_s^{k_m}\subsetneq \tilde Q^{k_m}$ we have that $Q_s^{k_m}$ is a dyadic maximal cube of the set $\{x: M_{\mathcal{D}}v(x)>a^{k_m}\}$. This means that
	\begin{equation}\label{ec2_demo_afirmacion3}
	\frac{1}{|Q_s^{k_m}|}\int_{Q_s^{k_m}}v(y)\,dy\leq 2^na^{k_m}\leq a^{k_m+1},
	\end{equation}
	since $a>2^n$, which leads us to $|Q_s^{k_m}\cap\{x:v(x)\leq a^{k_m+1}\}|>0$. Indeed, if it is not the case, denoting $E=Q_s^{k_m}\cap\{x: v(x)>a^{k_m+1}\}$ we would have that
	\[\frac{1}{|Q_s^{k_m}|}\int_{Q_s^{k_m}}v(y)\,dy=\frac{1}{|E|}\int_{Q_s^{k_m}}v(y)\,dy>\frac{1}{|E|}\int_E v(y)\,dy>a^{k_m+1},\]
	which contradicts~\eqref{ec2_demo_afirmacion3}. Therefore, $(k_m,s)\in \Gamma_N$ and $Q_s^{k_m}$ is contained in, at least, one principal cube. Let $Q_{\sigma}^k$ the smallest principal cube that contains 
	$Q_s^{k_m}$. By using conditions~\eqref{desigualdad1_conjuntoGn} and~\eqref{desigualdad2_conjuntoGn} we can write
	\[\frac{1}{|Q_j^\ell|}\int_{Q_j^\ell}u(y)\,dy>a^{(\ell-k)\beta }\frac{1}{|Q_\sigma^k|}\int_{Q_\sigma^k}u(y)\,dy\geq a^{(\ell-k_m)\beta }\frac{1}{|Q_s^{k_m}|}\int_{Q_s^{k_m}}u(y)\,dy.\]
	
	Also, from \eqref{desigualdad2_sucesionkm}
	\[\frac{1}{|\tilde Q^\ell|}\int_{\tilde Q^\ell}u(y)\,dy\leq \frac{2}{|\tilde Q^{k_m}|}\int_{\tilde Q^{k_m}}u(y)\,dy\leq 2[u]_{A_1}\inf_{\tilde Q^{k_m}}u\leq2[u]_{A_1}\frac{1}{|Q_s^{k_m}|}\int_{Q_s^{k_m}}u(y)\,dy.\]
	Combining these two estimates we get
	\begin{align*}
	\frac{1}{|Q_j^\ell|}\int_{Q_j^\ell}u(y)\,dy&>a^{(\ell-k)\beta }a^{(k-k_m)\beta }\frac{1}{|Q_s^{k_m}|}\int_{Q_s^{k_m}}u(y)\,dy\\
	&=a^{(\ell-k_m)\beta }\frac{1}{|Q_s^{k_m}|}\int_{Q_s^{k_m}}u(y)\,dy\\
	&>\frac{1}{2[u]_{A_1}}a^{(\ell-k_m)\beta }\frac{1}{|\tilde Q^\ell|}\int_{\tilde Q^\ell}u(y)\,dy,
	\end{align*}
	which completes the proof of the claim.
\end{proof}

\section{Interpolation and applications}\label{seccion: interpolacion y aplicaciones}

In this section we use some interpolation techniques that involve modular type inequalities in order to achieve the boundedness of the operator $M_\Phi$ in $L^p(w)$. It is known that $M_\Phi$ is bounded in $L^p(w)$ for every $r<p<\infty$ and every $w\in A_{p/r}$, when some properties of $\Phi$ are assumed. The result below was proved in \cite{B-D-P} in the more general context of Lebesgue spaces with variable exponent. 

\begin{teo}[\cite{B-D-P}, Thm. 2.5]
	Let $w$ be a weight, $1<p<\infty$ and $1\leq r<p$. Let $\varphi$ be a Young function that satisfies $\varphi\in B_\xi$, for every $\xi>r$. If $w\in A_{p/r}$, then $M_{\varphi}$ is bounded in $L^p(w)$.
\end{teo}

The main goal of this section is to give an alternative proof of the boundedness properties of $M_{\Phi}$, when $\Phi$ is a function that satisfies the hypotheses in Theorem~\ref{teorema principal}. It is not difficult to see that these hypotheses imply $\Phi\in B_\xi$ for every $\xi>r$, so $M_{\Phi}$ is bounded in $L^p(w)$ for every $r<p<\infty$ and $w\in A_{p/r}$. In order to achieve this estimate, we shall use the notation and an adaptation of the results of \cite{A-F}.

\begin{defi}\label{definicion casi creciente}
	We say that a function $\varphi:\R^+\to\R^+$ is \emph{quasi-increasing (q.i.)} if there exists a constant $\rho>0$ for which 
		\[\frac{1}{x}\int_0^x \varphi(t)\,dt\leq \rho\varphi(x), \]
	for every $x>0$. We will say that $\rho$ is a quasi-increasing constant of $\varphi$.
\end{defi}

\begin{defi}
	Let $\varphi, \psi:\R^+\to\R^+$ be functions. We will say that $\varphi\prec_N\psi$ if the collection $\{\psi(x)\varphi(\alpha/x)\}_{\alpha>0}$ is a family of q.i. functions with a constant independent of $\alpha$. 
\end{defi}

In view of Definition~\ref{definicion casi creciente}, if $\varphi\prec_N\psi$ then there exists a constant $\rho$ such that the inequality
	\[\frac{1}{x}\int_0^x \psi(t)\varphi(\alpha/t)\,dt\leq \rho\psi(x)\varphi(\alpha/x) \]
	holds for every $x>0$ and $\alpha>0$.

The next two lemmas will be useful to the main purpose of this section. Although both can be found in \cite{A-F}, we include the proofs for the sake of completeness.

\begin{lema}\label{implicacion de tipo debil para interpolacion modular}
	Let $\mu$ be a measure, $T$ a sub-additive operator, and $\varphi$ a Young function. Assume that
	\[\mu(\{x: |Tf(x)|>\lambda\})\leq C\int_{\R^n}\varphi\left(\frac{c|f(x)|}{\lambda}\right)\,d\mu,\]
	for some positive constants $C$ and $c$, and every $\lambda>0$. Also assume that $\norm{Tf}_{L^\infty(\mu)}\leq C_0\norm{f}_{L^\infty(\mu)}$. Then
	\[\mu\left(\left\{x: |Tf(x)|>\lambda\right\}\right)\leq C\int_{\{x: |f(x)|>\lambda/(2C_0)\}}\varphi\left(\frac{2cf(x)}{\lambda}\right)\,d\mu.\]
\end{lema} 

\begin{proof}
	Fix $\lambda>0$ and define $f_1=f\mathcal{X}_{\{x: |f(x)|>\lambda/(2C_0)\}}$ and $f_2=f-f_1$. So,
	\[
    \mu\left(\left\{x: |Tf(x)|>\lambda\right\}\right)\leq \mu\left(\left\{x: |Tf_1(x)|>\frac{\lambda}{2}\right\}\right)+\mu\left(\left\{x: |Tf_2(x)|>\frac{\lambda}{2}\right\}\right), 
    \]
    and observe that the second term is zero. Indeed, if $x$ satisfies $|Tf_2(x)|>\lambda/2$ we have
    \[\frac{\lambda}{2}<\norm{Tf_2}_{L^\infty(\mu)}\leq C_0\norm{f_2}_{L^\infty(\mu)},\]
    which is a contradiction. Thus, by applying the hypothesis to $f_1$ we get
    \begin{align*}
    \mu\left(\left\{x: |Tf(x)|>\lambda\right\}\right)&=\mu\left(\left\{x: |Tf_1(x)|>\frac{\lambda}{2}\right\}\right)\\
    &\leq C\int_{\R^n}\varphi\left(\frac{2c|f_1(x)|}{\lambda}\right)\,d\mu\\
    &\leq C\int_{\{x: |f(x)|>\lambda/(2C_0)\}}\varphi\left(\frac{2c|f(x)|}{\lambda}\right)\,d\mu. \qedhere
    \end{align*}
\end{proof}

\begin{lema}\label{lema interpolacion modular}
	Let $\mu$ be a measure and $\varphi:\R^+\to\R^+$ a non-decreasing function. Let $F$ and $G$ be nonnegative functions satisfying
	\begin{equation}\label{tipo debil para interpolacion modular}
	\mu(\{x\in \R^n: F(x)>\lambda\})\leq C\int_{\{x: G(x)>c\lambda\}}\varphi\left(\frac{G(x)}{\lambda}\right)\,d\mu,
	\end{equation}
for some positive constants $C,c$ and every $\lambda>0$. Let $\psi$ be a function in $\mathcal{C}^1([0,\infty))$ such that $\psi(0)=0$, $\psi'$ is non-decreasing and assume that $\varphi\prec_N \psi'$. Then, we have that
\[\int_{\R^n} \psi(F(x))\,d\mu\leq C \int_{\R^n} \psi\left(\frac{2G(x)}{c}\right)\,d\mu.\]
\end{lema}

\begin{proof}
	\begin{align*}
	\int_{\R^n}\psi(F(x))\,d\mu&=\int_0^\infty \psi'(\lambda)\mu(\{x: F(x)>\lambda\})\,d\lambda\\
	&\leq C\int_{\R^n}\left(\int_0^{c^{-1}G(x)}\psi'(\lambda)\varphi\left(\frac{G}{\lambda}\right)\,d\lambda\right)\,d\mu.
	\end{align*}
Now, since $\varphi\prec_N \psi'$ we can write
\[\int_0^{c^{-1}G(x)}\psi'(\lambda)\varphi\left(\frac{G}{\lambda}\right)\,d\lambda\leq \rho \varphi(c)c^{-1}G(x)\psi'(c^{-1}G(x)).\]
Observe also that 
\[\psi(x)=\int_0^x\psi'(t)\,dt\geq \int_{x/2}^x \psi'(t)\,dt\geq \frac{x}{2}\psi'\left(\frac{x}{2}\right),\]
since $\psi'$ is non-decreasing. Then,
\[\int_{\R^n}\psi(F(x))\,d\mu\leq C\int_{\R^n}\psi\left(\frac{2G}{c}\right)\,d\mu.\qedhere\]
\end{proof}

We are now in position to use Theorem~\ref{teorema principal} to give an alternative proof of the boundedness of $M_\Phi$. To begin with, fix $r<p<\infty$ and $w\in A_{p/r}$. Then, by the Jones' Factorization Theorem, there exist weights $u,v^r\in A_1$ such that $w=uv^{r(1-p/r)}=uv^{r-p}$. Thus, if $S_\Phi f=M_\Phi(fv)v^{-1}$  we can write
\begin{align*}
\int_{\R^n}M_{\Phi}(f)^pw\,dx&=\int_{\R^n}\left(\frac{M_{\Phi}(fv^{-1}v)}{v}\right)^puv^r\,dx\\
&=\int_{\R^n} S_{\Phi}(fv^{-1})^puv^r\,dx. 
\end{align*}
Note that, if we could prove that $S_{\Phi}$ is bounded in $L^p(uv^r)$, then 
\[\int_{\R^n} S_{\Phi}(fv^{-1})^puv^r\,dx\leq C \int_{\R^n} |f|^puv^{r-p}=C\int_{\R^n}|f|^pw\,dx.\]

Let us show the forementioned boundedness property of $S_\Phi$. We shall apply Lemma~\ref{lema interpolacion modular} with $\mu=uv^r\,dx$, $\varphi=\Phi$, $\psi(t)=t^p$, $F=S_\Phi(f)$ and $G=f$. In order to check the hypotheses of this lemma, we will prove some facts separately.

First, we shall show that $S_\Phi$ is bounded in $L^\infty(uv^r)$.

\begin{lema}\label{lema:  tipo infinito infinito fuerte de SPhi}
	Let $\Phi$, $S_\Phi$ as above and $u,v^r\in A_1$. Then, there exists a positive constant $C$ for which 
	\[\norm{S_{\Phi}f}_{L^\infty(uv^r)}\leq C \norm{f}_{L^\infty(uv^r)}.\]
\end{lema}

\begin{proof}
	 Let us first observe that $L^\infty(uv^r)=L^\infty$ because the sets that have null Lebesgue measure coincide with those which have $\mu$ measure zero, with $d\mu=uv^r\,dx$. So it will be enough to prove the lemma for the Lebesgue measure. We shall also assume that $\norm{f}_{L^\infty}=1$ and the general case will follow by homogeneity. 
	
	Since $v^r\in A_1$, there exists $\e>0$ such that $v^{r+\e}\in A_1$. For $t\geq t_0$, we have that $\Phi(t)\leq C_0t^r(\log t)^{\delta}\leq Ct^{r+\e}$. We want to estimate $M_\Phi(fv)v^{-1}$. Fix $x$ and a cube $Q$ containing $x$. Then, 	if  we take $\lambda=(|Q|^{-1}\int_Q v^{r+\e})^{1/(r+\e)}$,
	\begin{align*}
	\frac{1}{|Q|}\int_Q\Phi\left(\frac{fv}{\lambda}\right)\,dx&\leq\frac{C}{|Q|}\int_Q\Phi\left(\frac{v}{\lambda}\right)\,dx\\
	&= \frac{C}{|Q|}\left[\int_{Q\cap \{v\leq t_0\lambda\}}\Phi\left(\frac{v}{\lambda}\right)\,dx+\int_{Q\cap \{v>t_0 \lambda\}}\Phi\left(\frac{v}{\lambda}\right)\,dx\right]\\
	&\leq C\Phi(t_0) + \frac{C}{|Q|}\int_Q \left(\frac{v}{\lambda}\right)^{r+\e}\\
	&\leq C.
	\end{align*}
Therefore,
	\[\norm{fv}_{\Phi,Q}\leq C\lambda\leq C[v^{r+\e}]_{A_1}^{1/(r+\e)}v(x), \quad \textrm{ a.e. }x\in Q,\]
	and taking the supremum over all cube $Q$ containing $x$, we get $M_{\Phi}(fv)(x)v^{-1}(x)\leq C_0$. Finally, by taking the supremum over $x$ we are done.   
\end{proof}

Next, observe that condition \eqref{tipo debil para interpolacion modular} is guaranteed by combining Lemmas~\ref{lema:  tipo infinito infinito fuerte de SPhi} and \ref{implicacion de tipo debil para interpolacion modular}. It only remains to prove that $\Phi\prec_N \psi(t)=pt^{p-1}$. This fact is contained in the following lemma.


\begin{lema}
	Let $p>r$ and $\psi(t)=pt^{p-1}$. Then $\Phi\prec_N\psi$.
\end{lema}

\begin{proof}
	
	Observe first that $\Phi(t)\lesssim \Phi_0(t)$ for every $t>0$, where $\Phi_0(t)=t^r(1+\log^+t)^{\delta}$. Indeed, if $t\leq t_0$ we can use the fact that $\Phi\in \mathcal{F}_r$ to get 
\[\Phi(t)=\Phi\left(\frac{t}{t_0}t_0\right)\leq  C_r\left(\frac{t}{t_0}\right)^r\Phi(t_0)\leq C \Phi_0\left(t\right).\]
On the other hand, if $t>t_0$ we obtain
\[\Phi(t)\leq C_0 t^r(\log t)^{\delta}\leq C\Phi_0(t).\]
Therefore,
\begin{align*}
\frac{1}{x}\int_0^x\psi(t)\Phi\left(\frac{\alpha}{t}\right)\,dt&=\frac{1}{x}\s{k=0}{\infty}\int_{x/2^{k+1}}^{x/2^k}\psi(t)\Phi\left(\frac{\alpha}{t}\right)\,dt\\
&\leq \frac{1}{x}\s{k=0}{\infty}\int_{x/2^{k+1}}^{x/2^k}\psi(t)\Phi\left(\frac{2^{k+1}}{x}\alpha\right)\,dt\\
&\lesssim \frac{1}{x}\Phi\left(\frac{\alpha}{x}\right)\s{k=0}{\infty}\int_{x/2^{k+1}}^{x/2^k}\psi(t)\Phi(2^{k+1})\,dt
\end{align*}
Since $p>r$, there exists $\e>0$ such that $p-r>\e$. Recalling that $\psi(t)=pt^{p-1}$ and $\Phi_0(t)\leq C\nu^{-1}t^{r+\nu}$, for every $\nu>0$ and $t\geq 1$ we can estimate 
\begin{align*}
\frac{1}{x}\int_0^x\psi(t)\Phi\left(\frac{\alpha}{t}\right)\,dt&\lesssim \frac{1}{x}\Phi\left(\frac{\alpha}{x}\right)\s{k=0}{\infty}2^{(k+1)(r+\e)}\int_{x/2^{k+1}}^{x/2^k}pt^{p-1}\,dt\\
&\lesssim \frac{1}{x}\Phi\left(\frac{\alpha}{x}\right)\s{k=0}{\infty}2^{(k+1)(r+\e)}\left(\frac{x^p}{2^{kp}}-\frac{x^p}{2^{(k+1)p}}\right)\\
&\lesssim \frac{1}{x}\Phi\left(\frac{\alpha}{x}\right)\s{k=0}{\infty}2^{(k+1)(r+\e-p)}x^p\\
&\lesssim \psi(x)\Phi\left(\frac{\alpha}{x}\right).
\end{align*}

\end{proof}

\bigskip


\providecommand{\bysame}{\leavevmode\hbox to3em{\hrulefill}\thinspace}
\providecommand{\MR}{\relax\ifhmode\unskip\space\fi MR }
\providecommand{\MRhref}[2]{%
	\href{http://www.ams.org/mathscinet-getitem?mr=#1}{#2}
}
\providecommand{\href}[2]{#2}

\end{document}